\documentclass[12pt,reqno]{amsart}

\usepackage{amssymb,amsthm}
\usepackage{amsfonts,cmtiup,comment,stmaryrd}

\usepackage{mathrsfs,cmtiup}

\makeatletter
\def\blfootnote{\xdef\@thefnmark{}\@footnotetext}
\makeatother

\newtheorem{theorem}{Theorem}[section]

\newtheorem{lemma}[theorem]{Lemma}
\newtheorem{proposition}[theorem]{Proposition}
\newtheorem{corollary}[theorem]{Corollary}

\theoremstyle{definition}
\newtheorem{definition}[theorem]{Definition}
\newtheorem{remark}[theorem]{Remark}
\newtheorem*{definition*}{Definition}

\let\leq=\leqslant
\let\geq=\geqslant
\numberwithin{equation}{section}
\newcommand{\ed}{\end{document}}

\begin{document}
\title{Almost Engel compact groups}

\author{E. I. Khukhro}
\address{Sobolev Institute of Mathematics, Novosibirsk, 630090, Russia, and\newline University of Lincoln, U.K.}
\email{khukhro@yahoo.co.uk}

\author{P. Shumyatsky}

\address{Department of Mathematics, University of Brasilia, DF~70910-900, Brazil}
\email{pavel@unb.br}

\dedicatory{To Efim Zelmanov on occasion of his 60th birthday}


\begin{abstract}
We say that a group $G$ is almost Engel if for every $g\in  G$ there is a finite set ${\mathscr E}(g)$ such that for every $x\in G$ all sufficiently long commutators $[...[[x,g],g],\dots ,g]$ belong to ${\mathscr E}(g)$, that is, for every $x\in G$ there is a positive integer $n(x,g)$ such that $[...[[x,g],g],\dots ,g]\in {\mathscr E}(g)$ if $g$ is repeated at least $n(x,g)$ times. (Thus, Engel groups are precisely the almost Engel groups for which we can choose ${\mathscr E}(g)=\{ 1\}$ for all $g\in G$.)

We prove that if a compact (Hausdorff) group $G$ is almost Engel, then
$G$ has a finite normal subgroup $N$ such that $G/N$ is locally nilpotent. If in addition there is a uniform bound $|{\mathscr E}(g)|\leq m$ for the orders of the corresponding sets, then the subgroup $N$ can be chosen of order bounded in terms of $m$. The proofs use the Wilson--Zelmanov theorem saying that Engel profinite groups are locally nilpotent.
\end{abstract}
\maketitle

\section{Introduction}

A group $G$ is called an Engel group if for every $x,g\in G$ the equation $[x,g,g,\dots , g]=1$ holds, where $g$ is repeated in the commutator sufficiently many times depending on $x$ and $g$. (Throughout the paper, we use the left-normed simple commutator notation
$[a_1,a_2,a_3,\dots ,a_r]=[...[[a_1,a_2],a_3],\dots ,a_r]$.) A group is said to be locally nilpotent if every finite subset generates a nilpotent subgroup. Clearly, any locally nilpotent group is an Engel group. Wilson and Zelmanov \cite{wi-ze} proved the converse for profinite groups: any Engel profinite group is locally nilpotent. Later Medvedev \cite{med} extended this result to Engel compact (Hausdorff) groups.

 In this paper we consider {almost Engel} groups in the following precise sense.

 \begin{definition} \label{d}
 We  say that a group $G$ is \emph{almost Engel} if for every $g\in  G$ there is a \emph{finite} set ${\mathscr E}(g)$ such that for every $x\in G$ all sufficiently long commutators $[x,g,g,\dots ,g]$ belong to ${\mathscr E}(g)$, that is, for every $x\in G$ there is a positive integer $n(x,g)$ such that
 $$[x,\underbrace{g,g,\dots ,g}_n]\in {\mathscr E}(g)\qquad \text{for all }n\geq n(x,g).
 $$
 \end{definition}
\noindent Thus, Engel groups are precisely the almost Engel groups for which we can choose ${\mathscr E}(g)=\{ 1\}$ for all $g\in G$.

We prove that almost Engel compact groups are finite-by-(locally nilpotent). By a compact group we mean a compact Hausdorff topological group.

\begin{theorem}\label{t-e}
Suppose that $G$ is an almost Engel compact group.
Then $G$ has a finite normal subgroup $N$ such that $G/N$ is locally nilpotent.
\end{theorem}

In Theorem~\ref{t-e} it also follows that there is a locally nilpotent subgroup of finite index -- just consider $C_G( N)$.  The proof uses the aforementioned Wilson--Zelmanov theorem for profinite groups. First the case of a finite group $G$ is considered, where obviously the result  must be quantitative:  namely, given a uniform bound $ |{\mathscr E}(g)|\leq m$ for the cardinalities of the sets ${\mathscr E}(g)$ in the above definition, we prove that the order of the nilpotent residual $\gamma _{\infty}(G)=\bigcap _i\gamma _i(G)$ is bounded in terms of $m$ only. Then Theorem~\ref{t-e} is proved for profinite groups. Finally, the result for compact groups is derived with the use of the structure theorems for compact groups.

As in the case of finite groups, if there is a uniform bound $m$ for the cardinalities of the sets ${\mathscr E}(g)$ in  Theorem~\ref{t-e}, then the subgroup $N$ in the conclusion can be chosen to be of order bounded in terms of $m$ (Corollary~\ref{c-em}).

In an earlier paper \cite{khu-shu153} we obtained similar results about finite and profinite groups with a stronger condition of ``almost Engel'' type. That condition means that every element $g$ of the group is ``almost $n$-Engel'' for some $n=n(g)$ depending on $g$. Note, however, that Engel groups do not necessarily satisfy that condition. The new Definition~\ref{d} of almost Engel groups in the present paper imposes a weaker and more natural ``almost Engel'' condition and includes Engel groups (when all the subsets ${\mathscr E}(g)$ consist only of 1). Thus, the results of the present paper are stronger than in \cite{khu-shu153} even for (pro)finite groups, and  cover a wider class of compact groups.

First in \S\,\ref{s-1} we collect some elementary properties of minimal subsets ${\mathscr E}(g)$ in the definition of almost Engel groups. We deal with finite groups in \S\,\ref{s-f}, with profinite groups in \S\,\ref{s-pf}, and consider the general case of compact groups in \S\,\ref{s-comp}.

Our notation and terminology is standard; for profinite groups, see, for example, \cite{wilson}.

We say for short that an element $g$ of a group $G$ is an \textit{Engel element} if for any $x\in G$ we have $[x,g,g,\dots , g]=1$, where $g$ is repeated in the commutator sufficiently many times depending on $x$ (such elements $g$ are often called left Engel elements).

A subgroup (topologically) generated by a subset $S$ is denoted by $\langle S\rangle$. For a group $A$ acting by automorphisms on a group $B$ we use the usual notation for commutators $[b,a]=b^{-1}b^a$ and $[B,A]=\langle [b,a]\mid b\in B,\;a\in A\rangle$, and for centralizers $C_B(A)=\{b\in B\mid b^a=b \text{ for all }a\in A\}$ and $C_A(B)=\{a\in A\mid b^a=b\text{ for all }b\in B\}$.

Throughout the paper we shall write, say, ``$(a,b,\dots )$-bounded'' to abbreviate
``bounded above in terms of $a, b,\dots $ only''.

\section{Properties of Engel sinks}
\label{s-1}

Throughout this section we assume that $G$ is an almost Engel group in the sense of Definition~\ref{d}, so that for every $g\in  G$ there is a finite set ${\mathscr E}(g)$ such that for every $x\in G$
there is a positive integer $n(x,g)$ such that
\begin{equation} \label{e-def}
[x,\underbrace{g,g,\dots ,g}_n]\in {\mathscr E}(g)\qquad \text{for any }n\geq n(x,g).
\end{equation}
If ${\mathscr E}'(g)$ is another finite set with the same property for possibly different numbers $n'(x,g)$, then ${\mathscr E}(g)\cap {\mathscr E}'(g)$ also satisfies the same condition with the numbers $n''(x,g)=\max\{n(x,g),n'(x,g)\}$. Hence for every $g\in G$ there is a \emph{minimal} set satisfying the definition, which we again denote by ${\mathscr E}(g)$ and call the \emph{Engel sink for $g$}, or simply \emph{$g$-sink} for short. \emph{Henceforth we shall always use the notation ${\mathscr E}(g)$ to denote the (minimal) Engel sinks, and $n(x,g)$ the corresponding numbers satisfying \eqref{e-def}.}

For a fixed $g\in G$, consider the mapping of ${\mathscr E}(g)$ by the rule $z\to [z,g]$, which maps ${\mathscr E}(g)$  into itself by definition. By the minimality of ${\mathscr E}(g)$ this mapping is a permutation of ${\mathscr E}(g)$. Therefore we can speak of orbits (cycles)  of this permutation on ${\mathscr E}(g)$. It follows that every $z\in {\mathscr E}(g)$ can be represented in the form
\begin{equation}\label{e-orb1}
z=[z,\underbrace{g,\dots,g}_{k}]\qquad \text{for some }k\geq 1,
\end{equation}
and therefore also as
\begin{equation}\label{e-orb2}
z=[z,\underbrace{g,\dots,g}_{jk}]\qquad \text{for any positive integer }j.
\end{equation}
Conversely, elements satisfying \eqref{e-orb1} belong to ${\mathscr E}(g)$. We have thus proved the following.

\begin{lemma}\label{l-sink}
For any $g\in G$ the $g$-sink ${\mathscr E}(g)$ consists precisely of all elements $z$ such that $z=[z,{g,\dots,g}]$, where $g$ occurs at least once.
\end{lemma}

Clearly, every subgroup $H$ of $G$ is also an almost Engel group. Moreover,  by Lemma~\ref{l-sink},  for $h\in H$ the $h$-sink constructed within $H$ is precisely the subset ${\mathscr E}(h)\cap H$ of  the $h$-sink ${\mathscr E}(h)$ in $G$.
If $N$ is a normal subgroup of $G$, then $G/N$ is also an almost Engel group. For $\bar g=gN$ the $\bar g$-sink in $G/N$ is the image of  ${\mathscr E}(g)$ in this quotient group. These properties will be used throughout the paper without special references.

For any element $h$ of the centralizer $C_G(g)$ the equation $z=[z,g, \dots, g]$ implies    $z^h=[z^h,g^h, \dots, g^h]=[z^h,g, \dots, g]$. Hence ${\mathscr E}(g)$ is invariant under conjugation by $h$ by Lemma~\ref{l-sink}. If $|{\mathscr E}(g)|=m$, it follows that $h^{m!}$ centralizes ${\mathscr E}(g)$. We have thus proved the following.

\begin{lemma}\label{l-c-s}
If $h\in C_G(g)$ and $|{\mathscr E}(g)|=m$, then $h^{m!}$ centralizes ${\mathscr E}(g)$.
\end{lemma}

Engel sinks have especially nice properties in metabelian groups. We denote by $M'$ the derived subgroup of a group $M$.

\begin{lemma}\label{l-metab} Let $M$ be a metabelian almost Engel group.

{\rm (a)} Every sink ${\mathscr E}(g)$ is a normal subgroup contained in $M'$.

{\rm (b)} Elements of ${\mathscr E}(g)$ in the same orbit under the map $z\to [z,g]$ have the same order.
\end{lemma}

\begin{proof}  (a) By \eqref{e-orb1} and \eqref{e-orb2}, every element  $z\in {\mathscr E}(g)$ can be represented as $z=[z,g,\dots,g]$ with $g$ repeated $jk(z)$ times for $k(z)\geq 1$ and for every $j=1,2,\dots$. In particular,  ${\mathscr E}(g)\subseteq M'$. For $z_1,z_2\in {\mathscr E}(g)$ choose $j$ such that $jk(z_1)k(z_2)$ is larger than $n(z_1z_2,g)$. Then, by the standard metabelian laws,
$$
\begin{aligned}z_1z_2&=[z_1,{g,\dots,g}][z_2,{g,\dots,g}]\\
&= [z_1z_2,{g,\dots,g}]\in {\mathscr E}(g),\end{aligned}
$$
where $g$ is repeated ${jk(z_1)k(z_2)}$ times in each commutator.
Thus, the finite set ${\mathscr E}(g)$ is a subgroup.
For any $m\in M$, choose $jk(z)$ larger than $n([z,m],g)$. Then $[z,m]=[[z,g,\dots,g],m]=[[z,m],g,\dots,g]\in {\mathscr E}(g)$, where $g$ is repeated ${jk(z)}$ times in each commutator; this means that ${\mathscr E}(g)$ is a normal subgroup of $M$.

(b) Let $z\in {\mathscr E}(g)$. Since $z\in M'$, we have $[z^k,g]=[z,g]^k$ for any integer $k$. Hence the order of $[z,g]$ divides the order of $z$. Going in this way over the orbit, we return to $z$, which implies that the orders of all elements in the orbit are the same.
\end{proof}

\section{Finite almost Engel groups}
\label{s-f}

Of course, any finite group $G$ is almost Engel, and every element $g\in G$ has  finite (minimal) $g$-sink ${\mathscr E}(g)$. A meaningful result must be of quantitative nature, and this is what we prove in this section. The following theorem will also be used in the proof of the main results on profinite and compact groups.

\begin{theorem}\label{t-finite}
Let $G$ be a finite group, and $m$  a positive integer. Suppose that for every $g\in  G$ the cardinality of the $g$-sink ${\mathscr E}(g)$ is at most $m$.
 Then $G$ has a normal subgroup $N$ of order bounded in terms of $m$ such that $G/N$ is nilpotent.
\end{theorem}

The conclusion of the theorem can also be stated as a bound in terms of $m$ for the order of the nilpotent residual subgroup $\gamma _{\infty}(G)$, the intersection of all terms of the lower central series (which for a finite group is of course also equal to some subgroup $\gamma _n(G)$).

First we recall or prove a few preliminary results. We shall use the following  well-known properties of coprime actions: if $\alpha $ is an automorphism of a finite
group $G$ of coprime order, $(|\alpha |,|G|)=1$, then $ C_{G/N}(\alpha )=C_G(\alpha )N/N$ for any $\alpha $-invariant normal subgroup $N$, the equality $[G,\alpha ]=[[G,\alpha ],\alpha ]$ holds, and if $G$ is in addition abelian, then $G=[G,\alpha ]\times C_G(\alpha )$.

\begin{lemma}\label{l0}
Let  $P$ be a finite $p$-subgroup of a group $G$, and $g\in G$ a $p'$-element normalizing $P$. Then the order of  $[P,g]$ is bounded in terms of the cardinality of the $g$-sink ${\mathscr E}(g)$.
\end{lemma}

\begin{proof}
For the abelian $p$-group $V=[P,g]/[P,g]'$ we have $V=[V,g]$ and $C_V(g)=1$ because the action of $g$ on $V$ is coprime. Then $V=\{[v,g]\mid v\in V\}$ and therefore also
$$
V=\{[v,\underbrace{g,\dots ,g}_n\,]\mid v\in V\}
$$
for any $n$. Hence, $V$ is contained in the image of ${\mathscr E}(g)\cap [P,g]$ in $[P,g]/[P,g]'$, whence $|V|\leq |{\mathscr E}(g)|$.

Since $[P,g]$ is a nilpotent group, its order is bounded in terms of  $|[P,g]/[P,g]'|$ and its nilpotency class. We claim that, as a crude bound,  the nilpotency class of $[P,g]$ is at most  $2|{\mathscr E}(g)|+1$. Let $\gamma_i$ denote the terms of the lower central series of $[P,g]$. The number of factors of the lower central series of $[P,g]$ on which $g$ acts nontrivially is at most $|{\mathscr E}(g)|$, because for any such a factor $U=\gamma _i/\gamma _{i+1}$ we have
$$
1\ne  [U,g]=\{[u, \underbrace{g,\dots ,g}_n\,]\mid u\in U\}
$$
for any $n$, since the action of $g$ on $U$ is coprime, and therefore there  is an element of ${\mathscr E}(g)$  in $\gamma_i\setminus \gamma _{i+1}$. It remains to observe that $g$ cannot act trivially on two consecutive nontrivial factors of the lower central series of $[P,g]$. Indeed, if
$[\gamma_i, g]\leq \gamma _{i+1}$ and $[\gamma_{i+1}, g]\leq \gamma _{i+2}$, then by the Three Subgroup Lemma the inclusions $[\gamma_i, g, [P,g]]\leq [\gamma _{i+1}, [P,g]]=\gamma _{i+2}$ and $[[P,g],\gamma_{i}, g]=[ \gamma _{i+1},g]\leq \gamma _{i+2}$ imply the inclusion $[g, [P,g],\gamma _i]=[[P,g],\gamma _i]=\gamma _{i+1}\leq \gamma _{i+2}$, and the last inclusion implies that $\gamma_{i+1}=1$.
\end{proof}

The following lemma already appeared in \cite{khu-shu153}, but we reproduce the proof for the benefit of the reader.

\begin{lemma}\label{l2}
Let $V$ be an elementary abelian $q$-group, and $U$ a $q'$-group of
automorphisms of $V$. If $|[V,u]|\leq m$ for every $u\in U$, then $|[V,U]|$ is $m$-bounded, and therefore $|U|$ is also $m$-bounded.
\end{lemma}

\begin{proof}
First suppose that $U$ is abelian. We consider $V$ as an ${\mathbb F} _qU$-module. Pick $u_1\in U$ such that $[V,u_1]\ne 0$. By Maschke's theorem, $V= [V,u_1]\oplus C_V(u_1)$, and both summands are $U$-invariant, since $U$ is abelian. If $C_U([V,u_1])=1$, then $|U|$ is $m$-bounded and $[V,U]$ has $m$-bounded order being generated by $[V,u]$, $u\in U$. Otherwise pick $1\ne u_2\in C_U([V,u_1])$; then $V= [V,u_1] \oplus [V,u_2] \oplus C_V(\langle u_1,u_2\rangle )$. If $1\ne u_3\in C_U([V,u_1]\oplus [V,u_2])$, then $V= [V,u_1]\oplus [V,u_2]\oplus [V,u_3] \oplus C_V(\langle u_1,u_2,u_3\rangle )$, and so on. If $C_U([V,u_1]\oplus\dots \oplus [V,u_k])=1$ at some $m$-bounded step $k$, then again $[V,U]$ has $m$-bounded order. However, if there are too many steps, then
for the element $w=u_1u_2\cdots u_k$ we shall have $0\ne [V,u_i]= [[V,u_i],w]$, so that $[V,w] = [V,u_1]\oplus\dots \oplus [V,u_k]$ will have order greater than $m$, a contradiction.

We now consider the general case. Since every element $u\in U$ acts faithfully on $[V,u]$, the exponent of $U$ is $m$-bounded. If $P$ is a Sylow $p$-subgroup of $U$, let $M$ be a maximal normal abelian subgroup of $P$. By the above, $|[V,M]|$ is $m$-bounded. Since $M$ acts faithfully on $[V,M]$, we obtain that $|M|$ is $m$-bounded. Hence $|P|$ is $m$-bounded, since $C_P(M)= M$ and $P/M$ embeds in the automorphism group of $M$.  Since $|U|$ has only $m$-boundedly many prime divisors, it follows that $|U|$ is $m$-bounded. Since $[V,U]=\sum_{u\in U}[V,u]$, we obtain that $|[V,U]|$ is also $m$-bounded.
\end{proof}

 Recall that the Fitting series starts with the Fitting subgroup $F_1(G)=F(G)$, and by induction, $F_{k+1}(G)$ is the inverse image of $F(G/F_k(G))$. If $G$ is a  soluble group, then the least number $h$ such that $F_h(G)=G$ is the \textit{Fitting height} of $G$. The following lemma is well known and is easy to prove (see, for example, \cite[Lemma~10]{khu-maz}).

 \begin{lemma}\label{l-metan}
 If $G$ is a finite group of Fitting height 2, then $\gamma _{\infty}(G)=\prod _q [F_q,G_{q'}]$, where $F_q$ is a Sylow $q$-subgroup of $F(G)$, and $G_{q'}$ is a Hall ${q'}$-subgroup of $G$.
 \end{lemma}

We now approach the proof of Theorem~\ref{t-finite} with the following lemma.

\begin{lemma}\label{l3}
If $G$ is a finite group such that $|{\mathscr E}(g)|\leq m$ for all $g\in G$, then $G/F(G)$ has exponent at most $m!$.
\end{lemma}

\begin{proof}
Every element $g\in G$ centralizes all its powers. Therefore by Lemma~\ref{l-c-s}, since $|{\mathscr E}(g^{m!})|\leq m$ by hypothesis,  $g^{m!}$ centralizes ${\mathscr E}(g^{m!})$. By the minimality of the $g^{m!}$-sink, then ${\mathscr E}(g^{m!})=\{1\}$.  This means that
$g^{m!}$ is an Engel element and therefore belongs to the Fitting subgroup $F(G)$ by Baer's theorem \cite[Satz~III.6.15]{hup}.
\end{proof}

We are now ready to prove Theorem~\ref{t-finite}.

\begin{proof}[Proof of Theorem~\ref{t-finite}] Recall that $G$ is a finite group such that
$|{\mathscr E}(g)|\leq m$ for every $g\in G$. We need to show that $|\gamma _{\infty }(G)|$ is $m$-bounded.

First suppose that $G$ is soluble. Since $G/F(G)$ has $m$-bounded exponent by Lemma~\ref{l3}, the Fitting height of $G$ is $m$-bounded, which follows from the Hall--Higman theorems \cite{ha-hi}. Hence we can use induction on the Fitting height, with trivial base when the group is nilpotent and $\gamma _{\infty }(G)=1$. When the Fitting height is at least 2, consider the second Fitting subgroup $F_2(G)$. By Lemma \ref{l-metan} we have $\gamma _{\infty }(F_2(G))=\prod _q [F_q,H_{q'}]$, where $F_q$ is a Sylow $q$-subgroup of $F(G)$, and $H_{q'}$ is a Hall ${q'}$-subgroup of $F_2(G)$, the product taken over prime divisors of $|F(G)|$. For a given $q$, let $\bar H_{q'}=H_{q'}/C_{H_{q'}}(F_q)$, and let $V$ be the Frattini quotient $F_q/\Phi (F_q)$. Note that $\bar H_{q'}$ acts faithfully on $V$, since the action is coprime \cite[Satz~III.3.18]{hup}.

For every $x\in \bar H_{q'}$ the order  $|[V,x]|$ is $m$-bounded  by Lemma~\ref{l0}. Then $|\bar H_{q'}|$ is $m$-bounded by Lemma~\ref{l2}. As a result, $|[F_q , H_{q'}]|= |[F_q ,\bar H_{q'}]|$ is $m$-bounded, since $[F_q ,\bar H_{q'}]$ is the product of $m$-boundedly many subgroups $[F_q ,\bar h]$ for $h\in H_{q'}$, each of which has $m$-bounded order by Lemma~\ref{l0}.

For the same reasons, there are only $m$-boundedly many primes $q$ for which $[F_q , H_{q'}]\ne 1$. As a result, $|\gamma _{\infty }(F_2(G))|$ is $m$-bounded. Induction on the Fitting height applied to $G/\gamma _{\infty }(F_2(G))$ completes the proof in the case of soluble $G$.

Now consider the general case. Most of the following arguments follow the same scheme as in the proof of Theorem~1.2 in \cite{khu-shu153}. First we show that the quotient $G/R(G)$ by the soluble radical is of $m$-bounded order.
Let $E$ be the socle
of $G/R(G)$. It is known that $E$ contains its centralizer  in $G/R(G)$, so it suffices to show that $E$ has $m$-bounded order. In the quotient by the soluble radical, $E=S_1\times\dots\times S_k$ is a direct product of non-abelian finite simple groups $S_i$. Since the exponent of $G/F(G)$ is $m$-bounded by Lemma~\ref{l3}, the exponent of $E$ is also $m$-bounded. Now the classification of finite simple groups implies that every $S_i$ has $m$-bounded order, and it remains to show that the number of factors is also $m$-bounded. By Shmidt's theorem \cite[Satz~III.5.1]{hup}, every $S_i$ has a non-nilpotent soluble subgroup $R_i$, for which $\gamma _{\infty} (R_i)\ne 1$. Since we already proved our theorem for soluble groups, we can apply it to $T=R_1\times \dots \times R_k$. We obtain that $|\gamma _{\infty} (T)|$ is $m$-bounded, whence the number of factors is $m$-bounded.

Thus, $|G/R(G)|$ is $m$-bounded. Since $|\gamma _{\infty} (R(G)|$ is $m$-bounded by the soluble case proved above, we can consider $G/\gamma _{\infty} (R(G))$ and assume that $R(G)=F(G)$ is nilpotent. Then  $|G/F(G)|$ is $m$-bounded. We now use induction on $|G/F(G)|$.  The basis of this induction includes the trivial case $G/F(G)=1$ when $\gamma _{\infty}(G)=1$. But the bulk of the proof deals with the case where $G/F(G)$ is a non-abelian simple group.

 Thus, suppose that $G/F(G)$ is a non-abelian simple group of $m$-bounded order.
Let $g\in G$ be an arbitrary element. The subgroup $F(G)\langle g\rangle$ is soluble, and therefore $|\gamma _{\infty }(F(G)\langle g\rangle)|$ is $m$-bounded by the above. Since $\gamma _{\infty }(F(G)\langle g\rangle)$ is normal in $F(G)$, its normal closure $\langle \gamma _{\infty }(F(G)\langle g\rangle) ^G\rangle$ is a product of at most $|G/F(G)|$ conjugates, each normal in $F(G)$, and therefore has $m$-bounded order. Choose a transversal $\{t_1,\dots, t_k\}$ of $G$ modulo $F(G)$ and set
$$
K=\prod _i\langle \gamma _{\infty }(F(G)\langle t_i\rangle) ^G\rangle,
$$
which is a normal subgroup of $G$ of $m$-bounded order. It is sufficient to obtain an $m$-bounded estimate for $|\gamma _{\infty }(G/K)|$. Hence we can assume that $K=1$. We remark that then
\begin{equation}\label{e-nil}
[F(G), g, \dots , g]=1\qquad \text{for any } g\in G,
 \end{equation}
 when $g$ is repeated sufficiently many times. Indeed, $g\in F(G)t_i$ for some $t_i$, and the subgroup $F(G)\langle t_i\rangle$ is nilpotent due to our assumption that $K=1$.

We now claim that
\begin{equation}\label{e-nil2}
[F(G), G, \dots , G]=1
\end{equation}
 if $G$ is repeated sufficiently many times. It is sufficient to prove that $[F_q, G, \dots , G]=1$ for every Sylow $q$-subgroup $F_q$ of $F(G)$. For any $q'$-element $h\in G$ we have $[F_q,h]=[F_q,h,h]$ and therefore $[F_q,h]=1$ in view of \eqref{e-nil}. Let $H$ be the subgroup of $G$ generated by all $q'$-elements. Then $G=F_qH$ since $G/F(G)$ is non-abelian simple, and $[F_q,H]=1$, so that
$$
[F_q, G, \dots , G]=[F_q, F_q, \dots , F_q]=1
$$
for a sufficiently long commutator.

We finally show that $D:=\gamma _{\infty }(G)$ has $m$-bounded order. First we show that $D=[D,D]$. Indeed, since $G/F(G)$ is non-abelian simple, $D$ is nonsoluble and we must have
$$
G=F(G)[D,D].
$$
Taking repeatedly commutator with $G$ on both sides and applying \eqref{e-nil2}, we obtain $D=\gamma _{\infty }(G)\leq  [D,D]$, so $D=[D,D]$.

Since $F(G)\cap D$ is hypercentral in $D$ by \eqref{e-nil2} and $[D,D]=D$, it follows that
$F(G)\cap D\leq Z(D)\cap [D,D]$ by the well-known Gr\"un lemma \cite[Satz~4]{gr}. Thus, $D$ is a central covering of the simple group $D/(F(G)\cap D)\cong G/F(G)$, and therefore by Schur's theorem \cite[Hauptsatz~V.23.5]{hup} the order of $D$ is bounded in terms of the $m$-bounded order of $G/F(G)$.
Thus, we have proved that $|\gamma _{\infty }(G)|$ is $m$-bounded in the case where $G/F(G)$ is a non-abelian simple group.

We now finish the proof of Theorem~\ref{t-finite} by induction on the $m$-bounded order  $k=|G/F(G)|$ proving that $|\gamma _{\infty }(G)|$ is $(m,k)$-bounded. The basis of this induction is the case of  $G/F(G)$ being simple: nonabelian simple was considered above, and simple of prime order is covered by the soluble case. Now suppose that $G/F(G)$ has  a nontrivial proper normal subgroup with full inverse image $N$, so that $F(G)<N\lhd G$.  Since $F(N)=F(G)$, by induction applied to $N$, the order $|\gamma _{\infty }(N)|$ is bounded in terms of $m$ and $|N/F(G)|<k$. Since $N/\gamma _{\infty }(N)\leq F( G/\gamma _{\infty }(N))$,  by induction applied to $G/\gamma _{\infty }(N)$ the order $| \gamma _{\infty }(G/\gamma _{\infty }(N) )|$ is bounded in terms of $m$ and $|G/N|<k$. As a result, $|\gamma _{\infty }(G)| $ is $(m,k)$-bounded, as required.
\end{proof}

\section{Profinite almost Engel groups}\label{s-pf}

In this and the next sections, unless stated otherwise, a subgroup of a topological group will always mean a closed subgroup, all homomorphisms will be continuous, and quotients will be by closed normal subgroups. This also applies to taking commutator subgroups, normal closures, subgroups generated by subsets, etc. Of course, any finite subgroup is automatically closed. We also say that a subgroup is generated by a subset $X$ if it is generated by $X$ as a topological group.

In this section we prove Theorem~\ref{t-e} for profinite groups, while   Corollary~\ref{c-em} for profinite groups is an immediate corollary of Theorem~\ref{t-finite}.

\begin{theorem}\label{t-eprof}
Suppose that $G$ is an almost Engel profinite group.
Then $G$ has a finite normal subgroup $N$ such that $G/N$ is locally nilpotent.
\end{theorem}

Recall that pro-(finite nilpotent) groups, that is, inverse limits of finite nilpotent groups, are called \textit{pro\-nil\-po\-tent} groups.

\begin{lemma}\label{l-p-n}
An almost Engel  profinite group is pro\-nil\-po\-tent if and only if it is locally nilpotent.
\end{lemma}

\begin{proof}
Of course, any locally nilpotent profinite group is pro\-nil\-po\-tent.
Conversely, suppose that $G$ is an almost Engel  pro\-nil\-po\-tent group. We claim that all Engel sinks are trivial: ${\mathscr E}(g)=\{ 1\}$ for every $g\in G$. Indeed, otherwise  by Lemma~\ref{l-sink} ${\mathscr E}(g)$ contains a non-trivial element of the form $z=[z,g,\dots ,g]$ with $g$ occurring at least once. Choosing an open normal subgroup $N$ with nilpotent quotient $G/N$ such that $z\not\in N$, we obtain a contradiction.  Thus, ${\mathscr E}(g)=\{ 1\}$ for every $g\in G$, which means that all elements of $G$ are Engel elements, that is, $G$ is an Engel profinite group. Then $G$ is locally nilpotent by the Wilson--Zelmanov theorem \cite[Theorem~5]{wi-ze}.
\end{proof}

 Recall that the pro\-nil\-po\-tent residual of a profinite group $G$ is $\gamma _{\infty}(G)=\bigcap _i\gamma _i(G)$, where $\gamma _i(G)$ are  the terms of the lower central series; this is the smallest normal subgroup with pro\-nil\-po\-tent quotient. The following lemma is well known and is easy to prove. Here, element orders are understood as Steinitz numbers. The same results also hold in the special case of finite groups.

 \begin{lemma}\label{l-res}
 {\rm (a)} The pro\-nil\-po\-tent residual $\gamma _{\infty}(G)$ of a profinite group $G$ is equal to the subgroup generated by all commutators $[x,y]$, where $x,y$ are elements of coprime orders.

 {\rm (b)} For any normal subgroup $N$ of a profinite group $G$ we have $\gamma _{\infty}(G/N)=
 \gamma _{\infty}(G)N/N$.
 \end{lemma}

 \begin{proof}
 Part (a) follows from the characterization of pro\-nil\-po\-tent groups as profinite groups all of whose Sylow subgroups are normal. Part (b) follows from the fact that for any elements $\bar x, \bar y$ of coprime orders in a quotient $G/N$ of a profinite group $G$ one can find pre-images $x,y\in G$ which also have coprime orders.
 \end{proof}

The following generalization of Hall's criterion for nilpotency \cite{hall58}, which will be used later, already appeared in \cite{khu-shu153}, but we reproduce the proof for the benefit of the reader. We denote the derived subgroup of a group $B$ by $B'$.

\begin{proposition}\label{p-hall}

{\rm (a)} Suppose that $B$ is a normal subgroup of a group $A$ such that $B$ is nilpotent of class $c$ and $\gamma _{d}(A/B')$ is finite of order $k$.
Then the subgroup $C=C_A(\gamma _{d}(A/B'))=\{a\in A\mid [\gamma _{d}(A), a]\leq B'\}$ has finite $k$-bounded index and is nilpotent of $(c,d)$-bounded class.

{\rm (b)} Suppose that $B$ is a normal subgroup of a profinite group $A$ such that $B$ is pro\-nil\-po\-tent and $\gamma _{\infty}(A/B')$ is finite.
Then the subgroup $D=C_A(\gamma _{\infty }(A/B'))=\{a\in A\mid [\gamma _{\infty }(A), a]\leq B'\}$ is open and pro\-nil\-po\-tent.
\end{proposition}

\begin{proof}
(a)
Since $A/C$ embeds into $\operatorname{Aut}\gamma _{d}(A/B')$, the order of $A/C$ is $k$-bounded. We claim that $C$ is nilpotent of $(c,d)$-bounded class. Indeed, using simple-commutator notation for subgroups, we have
$$
[\underbrace{C,\dots ,C}_{d+1}, C, C,\dots ]\leq [[\gamma _d(A), C], C, \dots ]\leq [[B,B], C, \dots ],
$$
since $[\gamma _d(A), C]\leq B'$ by construction.
Applying repeatedly the Three Subgroup Lemma, we obtain
\begin{align*}
 [[B,B], \underbrace{C, \dots ,C}_{2d-1}, C, \dots ]&\leq \prod _{i+j=2d-1}[[B, \underbrace{C, \dots ,C}_{i}], [B,\underbrace{C, \dots ,C}_{j}], C, \dots ]\\
 &\leq
 [[[B,\underbrace{C, \dots ,C}_{d}], B], C, \dots ]\\
 &\leq [[[B,B],B],C,\dots ].
 \end{align*}
Thus, $\gamma _{d+1}(C)\leq \gamma _2(B)$, then $\gamma _{(d+1)+(2d-1)}(C)\leq \gamma _3(B)$, then a similar calculation gives \allowbreak $\gamma _{(d+1)+(2d-1)+(3d-2)}(C)\leq \gamma _4(B)$, and so on. An easy induction shows that $\gamma _{1+f(c,d)}(C)\leq \gamma _{c+1}(B)=1$ for $1+f(c,d)=1+dc(c+1)/2-c(c-1)/2$, so that $C$ is nilpotent of class $ f(c,d)$.

(b) As a centralizer of a normal section, $D$ is a closed normal subgroup. Since $A/D$ embeds into $\operatorname{Aut}\gamma _{\infty }(A/B')$, the subgroup $D$ has finite index; thus, $D$ is an open subgroup. We now show that the image of $D=C_A(\gamma _{\infty }(A/B'))$ in any finite quotient $\bar A$ of $A$ is nilpotent. Let bars denote the images in $\bar A$. Then $\gamma _{\infty }(\bar A/\bar B')=\overline{\gamma _{\infty }(A/B')}$ by Lemma~\ref{l-res}(b). Therefore, $\bar D\leq C_{\bar A}(\gamma _{\infty }(\bar A/\bar B'))$. In a finite group, $\gamma _{\infty }(\bar A/\bar B')=\gamma _{d}(\bar A/\bar B')$ for some positive integer $d$. Hence $\bar D$ is nilpotent by part (a).
\end{proof}

In general, the set of Engel elements in a profinite group may not be closed. But in an almost Engel group, Engel elements form a closed set, and moreover the following holds.

\begin{lemma}\label{l-closed}
Let $G$ be an almost Engel profinite group, and $k$ a positive integer. Then the  set
$$
E_{k}=\{x\in G\mid |{\mathscr E}(x)|\leq k\}.
$$
is closed in $G$.
\end{lemma}

\begin{proof}
We wish to show equivalently that the complement of $E_k$ is an open subset of $G$. Every element $g\in (G\setminus E_k)$ is characterized by the fact that $|{\mathscr E}(g)|\geq k+1$. Let $z_1,z_2,\dots ,z_{k+1}$ be some $k+1$ distinct elements in ${\mathscr E}(g)$. Using Lemma~\ref{l-sink}   we can write for every $i=1,\dots ,k+1$
\begin{equation}\label{e-open}
z_i=[z_i,g,\dots ,g], \quad \text{where } g \text{ is repeated } k_i\geq 1 \text{ times}.
 \end{equation}
 Let $N$ be an open normal subgroup of $G$ such that the images of $z_1,z_2,\dots ,z_{k+1}$ are distinct elements in $G/N$. Then equations \eqref{e-open} show that for any $u\in N$ the Engel sink $ {\mathscr E}(gu)$ contains an element in each of the $k+1$ cosets $z_iN$. Thus, all elements in the coset $gN$ are contained in $G\setminus E_k$. We have shown that every element of $G\setminus E_k$ has a neighbourhood that is also contained in $G\setminus E_k$, which is therefore an open subset of $G$.
\end{proof}

Recall that in Theorem~\ref{t-eprof} we need to show that an almost Engel  profinite group $G$ has a finite normal subgroup such that the quotient is locally nilpotent. The first step is to prove the existence of an open locally nilpotent subgroup.

\begin{proposition}\label{p-pf1}
If $G$ is an almost Engel profinite group, then it has an open
normal pro\-nil\-po\-tent subgroup.
\end{proposition}

Of course, the subgroup in question will also be locally nilpotent by Lemma~\ref{l-p-n}; the result can also be stated as the openness of the largest normal pro\-nil\-po\-tent subgroup.

\begin{proof}
For every $g\in G$ we choose an open normal subgroup $N_g$ such that ${\mathscr E}(g)\cap N_g=1$. Then $g$ is an Engel element in $N_g\langle g\rangle$. By Baer's theorem \cite[Satz~III.6.15]{hup}, in every finite quotient of $N_g\langle g\rangle$ the image of $g$ belongs to the Fitting subgroup. As a result, the subgroup $[N_g, g]$ is pro\-nil\-po\-tent.

Let $\tilde N_g$ be the normal closure of $[N_g, g]$ in $G$. Since $[N_g, g]$ is normal in $N_g$, which has finite index, $[N_g, g]$ has only finitely many conjugates, so $\tilde N_g$ is a product of finitely many normal subgroups of $N_g$, each of which is pro\-nil\-po\-tent.
Hence, so is $\tilde N_g$. Therefore all the subgroups $\tilde N_g$ are contained in the largest normal pro\-nil\-po\-tent subgroup $K$.

It is easy to see that $G/K$ is an $FC$-group (that is, every conjugacy class is finite): indeed, every $\bar g\in G/K$ is centralized by the image of $N_g$, which has finite index in $G$. A~profinite $FC$-group has finite derived subgroup \cite[Lemma~2.6]{sha}. Hence we can choose an open subgroup of $G/K$ that has trivial intersection with the finite derived subgroup of $G/K$ and therefore is abelian; let $H$ be its full inverse image in $G$. Thus, $H$ is an open subgroup such that the derived subgroup $H'$ is contained in $K$.

We now consider the metabelian quotient $M=H/K'$, which is also an almost Engel group, and temporarily use the symbols ${\mathscr E}(g)$ for the Engel $g$-sinks in $M$.
For every positive integer $k$, consider the set
$$
E_{k}=\{x\in M\mid |{\mathscr E}(x)|\leq k\}.
$$
By Lemma~\ref{l-closed}, every set $E_k$ is closed in $M$.
Since $M$ is an almost Engel group, we have
$$
M=\bigcup _{i} E_{i}.
$$
By the Baire category theorem \cite[Theorem~34]{kel}, one of these sets contains an open subset; that is, there is an open subgroup $U$ and a coset $aU$ such that $aU\subseteq E_{m}$ for some $m$. In other words, $|{\mathscr E}(au)|\leq m$ for all $u\in U$.

We claim that $|{\mathscr E}(u)|\leq m^2$ for any $u\in U$. Indeed, by Lemma~\ref{l-metab}(a) both ${\mathscr E}(a)$ and ${\mathscr E}(au)$ are normal subgroups of $M$ contained in $M'$. In the quotient
$$
\bar M=M/\big({\mathscr E}(a) {\mathscr E}(au)\big),
 $$
 both $\bar M'\langle \bar a\rangle$ and $\bar M' \langle \bar a \bar u\rangle$ are normal locally nilpotent subgroups. Hence their product, which contains $\bar u$, is also a locally  nilpotent subgroup by the Hirsch--Plotkin theorem \cite[12.1.2]{rob}. As a result, ${\mathscr E}(u)\leq {\mathscr E}(a) {\mathscr E}(au)$ and therefore $|{\mathscr E}(u)|\leq |{\mathscr E}(a)| \cdot |{\mathscr E}(au)|\leq m^2$.

 Since $M'\leq K/K'$,  it is easy to see that ${\mathscr E}(u) ={\mathscr E}(uk)$ for any $u\in U$ and any $k\in K/K'$. Therefore, setting $V=U(K/K')$ we have $|{\mathscr E}(v)|\leq m^2$ for any $v\in V$.

Thus, the Engel sinks of elements of $V$ uniformly satisfy the  inequality $|{\mathscr E}(v)|\leq m^2$ for all $v\in V$. The same inequality holds in every finite quotient $\bar V$ of $V$, to which we can therefore apply Theorem~\ref{t-finite}. As a result, $|\gamma _{\infty}(\bar V)|\leq n$ for some number $n=n(m)$ depending only on $m$. Then also $|\gamma _{\infty}( V)|\leq n$.

Let $W$ be the full inverse image of $V$, which is an open subgroup of $G$ containing $K$, and let $ \Gamma$ be the full inverse image of $\gamma _{\infty}( V)$.
Now let $F=C_W(\gamma _{\infty}( V))=\{w\in W\mid [\Gamma ,w]\leq K'\}$. By Proposition~\ref{p-hall}(b),
this is an open normal pro\-nil\-po\-tent subgroup, which completes the proof of Proposition~\ref{p-pf1}.
\end{proof}

\begin{proof}[Proof of Theorem~\ref{t-eprof}]  Recall that $G$ is an almost Engel  profinite group, and we need to show that $\gamma _{\infty}(G)$ is finite. Henceforth we denote by $F(L)$ the largest normal pro\-nil\-po\-tent subgroup of a profinite group $L$. By Proposition~\ref{p-pf1} we already know that $G$ has an open normal pro\-nil\-po\-tent
 subgroup, so that $F(G)$ is also open. Further arguments largely follow the scheme of proof of Theorem~1.1 in \cite{khu-shu153}.

Since $G/F(G)$ is finite, we can use induction on $|G/F(G)|$.
       The basis of this induction includes the trivial case $G/F(G)=1$ when $\gamma _{\infty}(G)=1$. But the bulk of the proof deals with the case where $G/F(G)$ is a finite simple group.

Thus, we assume that $G/F(G)$ is a finite simple group (abelian or non-abelian).
Let $p$ be a prime divisor of $|G/F(G)|$, and $g\in G\setminus F(G)$ an element of order $p^n$, where $n$ is either a positive integer or $\infty$ (so $p^n$ is a Steinitz number). For any prime $q\ne p$, the element $g$ acts by conjugation on the Sylow $q$-subgroup $Q$ of $F(G)$ as an automorphism of order dividing $p^n$. The subgroup $[Q, g]$ is a normal subgroup of $Q$ and therefore also a normal subgroup of $F(G)$. The image of $[Q, g]$ in any finite quotient has order bounded in terms of $|{\mathscr E}(g)|$ by Lemma~\ref{l0}. It follows that $[Q, g]$ is finite of order bounded in terms of $|{\mathscr E}(g)|$.

Since $[Q, g]$ is normal in $F(G)$, its normal closure $\langle [Q, g]^G\rangle $ in $G$ is a product of finitely many conjugates and is therefore also finite. Let $R$ be the product of these closures $\langle [Q, g]^G\rangle $ over all Sylow $q$-subgroups $Q$ of $F(G)$ for $q\ne p$. Since $[Q, g]$ is finite of order bounded in terms of $|{\mathscr E}(g)|$ as shown above, there are only finitely many primes $q$ such that $[Q,g]\ne 1$ for the Sylow $q$-subgroup $Q$ of $F(G)$. Therefore $R$ is finite, and it is sufficient to prove that $\gamma _{\infty }(G/R)$ is finite. Thus, we can assume that $R=1$. Note that then $[Q, g^a]=1$ for any conjugate $g^a$ of $g$ and any Sylow $q$-subgroup $Q$ of $F(G)$ for $q\ne p$.

Choose a transversal $\{t_1,\dots, t_k\}$ of $G$ modulo $F(G)$.
 Let $G_1=\langle g^{t_1}, \dots ,g^{t_k}\rangle$. Clearly, $G_1F(G)/F(G)$ is generated by the conjugacy class of the image of $g$. Since $G/F(G)$ is simple, we have $G_1F(G)=G$. By our assumption, the Cartesian product $T$ of all Sylow $q$-subgroups of $F(G)$ for $q\ne p$ is centralized by all elements $g^{t_i}$. Hence,  $[G_1, T]=1$. Let $P$ be the Sylow $p$-subgroup of $F(G)$ (possibly, trivial). Then also $[PG_1, T]=1$, and therefore
 $$\gamma _{\infty }(G)=\gamma _{\infty }(G_1F(G))= \gamma _{\infty }(PG_1).$$
  The image of $\gamma _{\infty }(PG_1)\cap T$ in $G/P$ is contained both in the centre and in the derived subgroup of $PG_1/P$ and therefore is isomorphic to a subgroup of the Schur multiplier of the finite group $G/F(G)$. Since the Schur multiplier of a finite group is finite \cite[Hauptsatz~V.23.5]{hup}, we obtain that $\gamma _{\infty }(G)\cap T=\gamma _{\infty }( PG_1)\cap T$ is finite. Therefore  we can assume that $T=1$, in other words, that $F(G)$ is a $p$-group.

 If $|G/F(G)|=p$, then $G$ is a pro-$p$ group, so it is pro\-nil\-po\-tent, which means that $\gamma _{\infty }( G)=1$ and the proof is complete. If $G/F(G)$ is a non-abelian simple group, then we choose another prime $r\ne p$ dividing $|G/F(G)|$ and repeat the same arguments as above with $r$ in place of $p$. As a result, we reduce the proof to the case $F(G)=1$, where the result is obvious.

We now finish the proof of Theorem~\ref{t-eprof} by induction on  $|G/F(G)|$. The basis of this induction where $G/F(G)$ is a simple group was proved above. Now suppose that $G/F(G)$ has  a nontrivial proper normal subgroup with full inverse image $N$, so that $F(G)<N\lhd G$. Since $F(N)=F(G)$, by induction applied to $N$ the group $\gamma _{\infty }(N)$ is finite. Since $N/\gamma _{\infty }(N)\leq F( G/\gamma _{\infty }(N))$,  by induction applied to $G/\gamma _{\infty }(N)$ the group $ \gamma _{\infty }(G/\gamma _{\infty }(N) )$ is also finite. As a result, $\gamma _{\infty }(G) $ is finite, as required.
\end{proof}

\section{Compact almost Engel groups}
\label{s-comp}
In this section we prove the main Theorem~\ref{t-e}
about compact almost Engel groups.  We use the structure theorems for compact groups and the results of the preceding section on profinite almost Engel groups.

Recall that a group $H$ is said to be \emph{divisible} if for every $h\in H$ and every positive integer $k$ there is an element $x\in H$ such that $x^k=h$.

\begin{proposition}\label{pr-div}
If $H$ is an almost Engel divisible group, then $H$ is in fact an Engel group.
\end{proposition}

\begin{proof}
We need to show that ${\mathscr E}(h)=\{ 1\}$ for every $h\in H$.
Let $|{\mathscr E}(h)|=k$.  Let $g\in H$ be an element such that $g^{k!}=h$. Since $g$ centralizes $h$, by Lemma~\ref{l-c-s} we obtain that $h=g^{k!}$ centralizes ${\mathscr E}(h)$. By the minimality of the $h$-sink, then ${\mathscr E}(h)=\{1\}$, as required.
\end{proof}

We are ready to prove Theorem~\ref{t-e}.

\begin{proof}[Proof of Theorem~\ref{t-e}] Let $G$ be an almost Engel compact  group; we need to show that there is a finite subgroup $N$ such that $G/N$ is locally nilpotent.
 By the well-known structure theorems (see, for example, \cite[Theorems~9.24 and 9.35]{h-m}), the connected component of the identity $G_0$ in $G$ is a divisible group such that $G_0/Z(G_0)$ is a Cartesian product of simple compact Lie groups, while the quotient $G/G_0$ is a profinite group. By Proposition~\ref{pr-div}, $G_0$ is  an Engel group. Compact Lie groups are linear groups, and linear Engel groups are locally nilpotent by the Garashchuk--Suprunenko theorem \cite{gar-sup} (see also \cite{grb}). Hence $G_0=Z(G_0)$ is an abelian subgroup.

 \begin{remark}
 If a compact group $G$ is an Engel group, then by the above $G$ is an extension of an abelian subgroup $G_0$ by a profinite group $G/G_0$. Being an Engel group, $G/G_0$ is locally nilpotent by the Wilson--Zelmanov theorem \cite[Theorem~5]{wi-ze}. It is known that an Engel abelian-by-(locally nilpotent) group is locally nilpotent (see, for example, \cite[12.3.3]{rob}). This gives an alternative proof of Medvedev's theorem \cite{med}.
 \end{remark}

We proceed with the proof of Theorem~\ref{t-e}.

\begin{lemma}\label{l-eng}
For every $g\in G$ we have ${\mathscr E}(g)\cap G_0=1$, which is equivalent to the fact that for any $x\in G_0$ we have
$$
[x,g,\dots ,g]=1
$$
if $g$ is repeated sufficiently many times.
\end{lemma}

\begin{proof}
Suppose the opposite and choose $1\ne z\in {\mathscr E}(g)\cap G_0$. By Lemma~\ref{l-sink} then $z=[z,g,\dots, g]$ with $g$ occurring at least once. Hence  $z$ belongs also to the $g$-sink within the semidirect product $G_0\langle g\rangle$. Since this subgroup is metabelian, by Lemma~\ref{l-metab}(a) the $g$-sink in $G_0\langle g\rangle$ is a subgroup, which is equal to ${\mathscr E}(g)\cap G_0$.  Therefore we can choose $z_1$ in ${\mathscr E}(g)\cap G_0$ of some prime order $p=|z_1|$.  Again by Lemma~\ref{l-sink} we have $z_1=[z_1,g,\dots, g]$ with $g$ occurring at least once.
In the divisible group $G_0$ for every $k=1,2,\dots $ there is an element $z_k$ such that $z_k^{p^k}=z_1$. We have $y_k= [z_k, g,\dots ,g]\in {\mathscr E}(g)\cap G_0$ when $g$ is repeated sufficiently many times.
 Then $y_k^{p^k}= [z_k^{p^k}, g,\dots ,g]=[z_1,g,\dots, g]$, which is an element of the orbit of $z_1$ in ${\mathscr E}(g)\cap G_0$ under the mapping $x\to [x,g]$ and therefore has the same order $p=|z_1|$ by Lemma~\ref{l-metab}(b). Thus, $y_k$ is an element of ${\mathscr E}(g)$  of order exactly $p^{k+1}$, for $k=1,2,\dots $. As a result, ${\mathscr E}(g)$ is infinite, a contradiction with the hypothesis.
\end{proof}

Applying Theorem~\ref{t-eprof} to the almost Engel profinite group $\bar G=G/G_0$ we obtain a finite normal subgroup $D$ with locally nilpotent quotient. Then $D$ contains all Engel sinks $\bar{\mathscr E}(g)$ of elements $g\in\bar G$ and therefore the subgroup $E$ generated by them:
$$
E:=\langle \bar{\mathscr E}(g)\mid g\in \bar G\rangle\leq D.
$$
Clearly, $\bar{\mathscr E}(g)^h=\bar{\mathscr E}(g^h)$ for any $h\in\bar G$; hence  $E$ is a normal finite subgroup of $\bar G$. We replace $D$ by $E$ in the sense that $\bar G/E$ is also locally nilpotent by the Wilson--Zelmanov theorem \cite[Theorem~5]{wi-ze}, since this is an Engel profinite group.

We now consider the action of  $\bar G$ by automorphisms  on $G_0$ induced by conjugation.

\begin{lemma}\label{l-central}
The subgroup $E$ acts trivially on $G_0$.
\end{lemma}

\begin{proof}
The abelian divisible group $G_0$ is a direct product $A_0\times\prod _pA_p$ of a torsion-free divisible group $A_0$ and Sylow subgroups $A_p$ over various primes $p$. Clearly, every Sylow subgroup is normal in $G$.

First we show that $E$ acts trivially on each $A_p$. It is sufficient to show that for every $g\in \bar G$ every element  $z\in \bar{\mathscr E}(g)$ acts trivially on $A_p$. Consider the action of $\langle z, g\rangle$ on $A_p$. Note that $\langle z, g\rangle=\langle z^{\langle  g\rangle}\rangle\langle  g\rangle$, where $\langle z^{\langle  g\rangle}\rangle$ is  a finite $g$-invariant subgroup, since it is contained in the finite subgroup $E$. For any $a\in A_p$ the subgroup
$$
\langle a^{\langle g\rangle}\rangle=\langle a,[a,g], [a,g,g],\dots \rangle
$$
is a finite $p$-group by Lemma~\ref{l-eng}, and this subgroup is  $g$-invariant.
 Its images under the action of elements of the finite group  $\langle z^{\langle  g\rangle}\rangle$  generate a finite $p$-group $B$, which is $\langle z, g\rangle$-invariant. Lemma~\ref{l-eng} implies that the image of $\langle z, g\rangle$  in its action on $B$ must be a $p$-group. Indeed, otherwise this image contains a $p'$ element $y$ that acts non-trivially on the Frattini quotient $V=B/\Phi (B)$. Then  $V=[V,y]$ and $C_V(y)=1$, whence  $[V,y]=\{[v,g]\mid v\in [V,y]\}$ and therefore also $[V,y]=\{[v,{y,\dots ,y}]\mid v\in [V,y]\} $ with $y$ repeated $n$ times, for any $n$, contrary to Lemma~\ref{l-eng}. But since $z$ is an element of  $\bar{\mathscr E}(g)$, by Lemma~\ref{l-sink} we have  $z=[z,g,\dots ,g]$ with at least one occurrence of $g$. Since a finite $p$-group is nilpotent, this implies that the image of $z$ in its action on $B$ must be trivial. In particular, $z$ centralizes $a$. As a result $E$ acts trivially on $A_p$, for every prime $p$.

 We now show that $E$ also acts trivially on the quotient $W=G_0/\prod _pA_p$ of $G_0$ by its torsion part. Note that $W$ can be regarded as a vector space over ${\mathbb Q}$. Every element $g\in E$ has finite order and therefore by Maschke's theorem   $W=[W,g]\times C_W(g)$. If $[W,g]\ne 1$, then $[W,g]=\{[w,{g,\dots ,g}]\mid w\in [W,g]\} $ with $g$ repeated $n$ times, for any $n$. This contradicts Lemma~\ref{l-eng}.

 Thus, $E$ acts trivially both on $W$ and on  $\prod _pA_p$. Then any automorphism of $G_0$ induced by conjugation by $g\in E$ acts on every element $a\in A_0$ as $a^g=at$, where $t=t(a)$ is an element of finite order in $G_0$. Since $a^{g^i}=at^i$, the order of $t$ must divide the order of $g$. Assuming the action of $E$ on $G_0$ to be non-trivial, choose an element $g\in E$  acting on $G_0$ as an automorphism of some prime order $p$. Then there is $a\in A_0$ such that $a^g=at$, where $t\in G_0$ has order $p$. For any $k=1,2,\dots $ there is an element $a_k\in A_0$ such that $a_k^{p^k}=a$. Then $a_k^g=a_kt_k$, where $t_k^{p^k}=t$. Thus $|t_k|=p^{k+1}$, and therefore  $p^{k+1}$ divides the order of $g$, for every $k=1,2,\dots$. We arrived  at a contradiction since $g$ has finite order.
\end{proof}

Let $F$ be the full inverse image of $E$ in $G$. Recall that we have normal subgroups $G_0\leq F\leq G$ such that $G_0$ is divisible, $F/G_0$ is finite, and $G/F$ is locally nilpotent. We aim at producing a finite normal subgroup $N$ such that $G/N$ is locally nilpotent.

 By Lemma~\ref{l-central} the subgroup $G_0$ is contained in the centre of the full inverse image $F$ of $E$, so that $F$ has centre of finite index. Then the derived subgroup $F'$ is finite by Schur's theorem \cite[Satz~IV.2.3]{hup}. We can assume that $F'=1$, so that then $F$ is abelian. In the abstract abelian group $F$ the divisible subgroup $G_0$ has a complement $C$, which is obviously finite, $F=G_0\times C$. Let $M$ be the  normal closure of $C$ in $G$.

 Note that $G/M$ is a locally nilpotent group. Indeed, $G/F$ is locally nilpotent, while $F/M$ is $G$-isomorphic to $G_0/(G_0\cap M)$ so that $G/M$ is an Engel group by Lemma~\ref{l-eng}. Being an abelian-by-(locally nilpotent) Engel group, then $G/M$ is locally nilpotent.

  Consider the natural semidirect product $M\rtimes (G/C_G(M))$ with the induced topology, in which the action of $G/C_G(M)$ on $M$ is continuous. Both $M$ and $G/C_G(M)$ are profinite groups; therefore  $M\rtimes (G/C_G(M))$ is also a profinite group (see \cite[Lemma 1.3.6]{wilson}).
  Note that $ G/C_G(M)$ is locally nilpotent, since $M\leq C_G(M)$.

 The group $M\rtimes (G/C_G(M))$ is also almost Engel, so by Theorem~\ref{t-eprof} it contains a finite normal subgroup with locally nilpotent quotient. This finite subgroup therefore contains the subgroup $K$ generated by all Engel sinks in $M\rtimes (G/C_G(M))$; since  $ G/C_G(M)$ is locally nilpotent, $K\leq M$.
 Since $G/M$ is also locally nilpotent, the Engel sinks in $G$ are all contained in $M$ and coincide with the Engel sinks in $M\rtimes (G/C_G(M))$. Renaming $K$ by $N$ as a subgroup of $G$ we arrive at the required result. Indeed, $N$ is a normal subgroup because  ${\mathscr E}(g)^h={\mathscr E}(g^h)$ for any $h\in G$. The  group $G/N$ is locally nilpotent being an Engel group which is an extension of an abelian group $F/N$ by a locally nilpotent group $G/F$. The proof of Theorem~\ref{t-e} is complete.
 \end{proof}

\begin{corollary}
\label{c-em}
Let $G$ be an almost Engel  compact group such that for some positive integer  $m$  all Engel sinks ${\mathscr E}(g)$ have cardinality at most $m$. Then $G$ has a finite normal subgroup $N$ of order bounded in terms of $m$ such that $G/N$ is locally nilpotent.
\end{corollary}

\begin{proof}
By Theorem~\ref{t-e} the group $G$ is finite-by-(locally nilpotent). Therefore every abstract finitely generated subgroup $H$ of $G$ is finite-by-nilpotent and residually finite. By Theorem~\ref{t-finite}, every finite quotient of $H$ has nilpotent residual of  $m$-bounded order. Hence $\gamma _{\infty}(H)$ is also finite of $m$-bounded order, and $H/\gamma _{\infty}(H)$ is nilpotent. Thus, every finitely generated subgroup of $G$ has a normal subgroup of $m$-bounded order with nilpotent quotient. By the standard inverse limit argument, the group $G$ has a normal subgroup of $m$-bounded order with locally nilpotent quotient.
\end{proof}

\section*{Acknowledgements}
The authors thank Gunnar Traustason and John Wilson for stimulating discussions.

The first author was supported  by the Russian Science Foundation, project no. 14-21-00065,
and the second by the Conselho Nacional de Desenvolvimento Cient\'{\i}fico e Tecnol\'ogico (CNPq), Brazil. The first author thanks  CNPq-Brazil and the University of Brasilia for support and hospitality that he enjoyed during his visits to Brasilia.


\begin{thebibliography}{99}

\bibitem{gar-sup}
 M. S. Garashchuk and D. A. Suprunenko,  Linear nilgroups, \textit{Dokl. Akad. Nauk BSSR} {\bf 4} (1960), 407--408. (Russian)

\bibitem{grb}
K. W. Gruenberg, The Engel structure of linear groups, \textit{J.~Algebra} {\bf  3} (1966), 291--303.

\bibitem{gr} O. Gr\"un, Beiträge zur Gruppentheorie. I, \textit{J.~Reine Angew. Math.} \textbf{174} (1935), 1--14.

\bibitem{hall58} P. Hall, Some sufficient conditions for a
group to be nilpotent,
{\it Illinois J. Math.} {\bf 2} (1958), 787--801.

\bibitem{ha-hi} P. Hall and G. Higman, The $p$-length of a $p$-soluble group and reduction
theorems for Burnside's problem, \emph{Proc. London Math. Soc. (3)} {\bf
 6} (1956), 1--42.

\bibitem{h-m} K. H. Hofmann and S. A. Morris,
The structure of compact groups, De Gruyter, Berlin, 2006.

\bibitem{hup} B. Huppert, {\it Endliche Gruppen}. I
(Springer, Berlin, 1967).

\bibitem{kel} J. L. Kelley, {\it General topology}, Grad. Texts in Math., vol.~27, Springer, New York,  1975.

\bibitem{khu-maz}
E. I. Khukhro and V. D. Mazurov,
Finite groups with an automorphism of prime order
whose centralizer has small rank,
\emph{J.~Algebra} {\bf 301} (2006) 474--492.

\bibitem{khu-shu153} E. I. Khukhro and P. Shumyatsky, Almost Engel finite and profinite groups, \emph{Int. J. Algebra Comput.}, {\bf  26}, no.~5 (2016), 973--983.


\bibitem{med} Yu. Medvedev,  On compact Engel groups,  Israel J. Math. {\bf 185} (2003), 147--156.

\bibitem{rob} D. J. S. Robinson, A course in the theory of groups, Springer, New York, 1996.

\bibitem{sha} A.~Shalev,
Profinite Groups with Restricted Centralizers, \emph{Proc. Amer. Math. Soc.} \textbf{122}, no.~4 (1994), 1279--1284.

\bibitem{wilson} J. S. Wilson, \emph{Profinite groups} (Clarendon Press, Oxford, 1998).

\bibitem{wi-ze} J. S. Wilson and E. I. Zelmanov, Identities for Lie algebras of pro-$p$ groups, \emph{J. Pure Appl. Algebra} {\bf 81}, no.~1 (1992), 103--109.

\end{thebibliography}
\end{document}